\documentclass[11pt,letterpaper]{amsart}
\usepackage[utf8]{inputenc}
\usepackage{fancyhdr}
\usepackage{amsfonts}
\usepackage{amsmath}
\usepackage[active]{srcltx}
\usepackage{amsmath,amssymb}
\usepackage{mathtools}
\usepackage{longtable}
\usepackage[arrow,matrix,curve]{xy}
\usepackage{amstext}
\usepackage{amssymb}
\usepackage{amsfonts}
\usepackage{amssymb}
\usepackage{amsmath}
\usepackage{graphicx}
\usepackage{afterpage,float,amsmath}
\usepackage{graphicx}
\usepackage{enumitem}
\usepackage{upgreek}
\usepackage{tikz}
\usepackage{multirow}
\usepackage{caption}
\usepackage{mathrsfs}
\usepackage{amsthm}
\usepackage{hyperref}
\usepackage{todonotes}
\usepackage{tikz-cd}
\usepackage{tikz}
\usepackage{ulem}
\usepackage{tkz-graph}

\theoremstyle{plain}
\newtheorem{theorem}{Theorem}[section]
\newtheorem{lemma}[theorem]{Lemma}
\newtheorem{proposition}[theorem]{Proposition}
\newtheorem{corollary}[theorem]{Corollary}
\newtheorem{definition}[theorem]{Definition}
\newtheorem{remark}[theorem]{Remark}
\newtheorem{example}[theorem]{Example}
\newtheorem{conjecture}[theorem]{Conjecture}

\newcommand{\Jac}{\mathrm{Jac}}
\newcommand{\rank}{\mathrm{rank}~}

\newcommand{\Hess}{\mathrm{Hess}}
\newcommand{\hess}{\mathrm{hess}}

\newcommand{\Ann}{\mathrm{Ann}}
\newcommand{\codim}{\mathrm{codim}}
\newcommand{\rk}{\mathrm{rk}}

\newcommand\K{\mathbb{K}}

\title{On higher Jacobians, Laplace equations and Lefschetz properties}
\author{Charles Almeida}
\address{ICEx - UFMG \\
Department of Mathematics,   Av. Ant\^onio Carlos, 6627\\
30123-970 Belo Horizonte, MG, Brazil}
\email{charlesalmeida@mat.ufmg.br}
\author{Aline V. Andrade}
\address{ICEx - UFMG \\
Department of Mathematics,   Av. Ant\^onio Carlos, 6627\\
30123-970 Belo Horizonte, MG, Brazil}
\email{andradealine@mat.ufmg.br}
\author{Rodrigo Gondim}
\address{Universidade Federal Rural de Pernambuco,\\
Av. Don Manoel de Medeiros s/n, Dois Irmãos - Recife - PE 52171-900, Brasil}
\email{rodrigo.gondim@ufrpe.br}

\begin{document}

\maketitle

\begin{abstract}
	Let $A$ be a standard graded $\K$-algebra of finite type over an algebraically closed field of characteristic zero. We use apolarity to construct, for each degree $k$, a projective variety whose osculating defect in degree $s$ is equivalent to the non maximality of the rank of the multiplication map for a power of a general linear form $\times L^{k-s}: A_s \to A_k$.
	In the Artinian case,  this notion corresponds  to the failure of the Strong Lefschetz property for $A$, which allows to reobtain some of the foundational theorems in the field. It also implies the SLP for codimension two Artinian algebras, a known result. The results presented in this work provide new insights on the geometry of monomial Togliatti systems, and offer a geometric interpretation of the vanishing of higher order Hessians.
	\medskip
   	 
   	 
   	 \medskip
   	 
\end{abstract}

\color{black}{.}
\section{Introduction}

Natural axioms of cohomology show that, in many categories, the cohomology ring of an object is an Artinian algebra satisfying Poincaré duality. Assuming they are commutative rings, these algebras can be characterized as Artinian Gorenstein algebras (see \cite{MW} and \cite[Theorem 2.71]{HMMNWW} ). Inspired by the Hard Lefschetz theorem for smooth complex projective varieties there is a contemporary effort to define smooth objects in certain categories in order to obtain a version of Lefschetz theorem among other related results. For instance, consider K\"ahler manifolds, simplicial complexes, convex polytopes, Coxeter groups, tropical varieties, and matroids (see \cite{Be,HL,Ka,KN,NW,St,St2}).   \\

In a purely algebraic context, Stanley and Watanabe introduced an abstraction called the weak and the strong Lefschetz properties (WLP and SLP, respectively) for any Artinian $\K$-algebra (see \cite{St,Wa}). In the Gorenstein case, Maeno and Watanabe, see \cite{MW}, gave a criterion for the failure of the SLP introducing the higher order Hessians. This criterion was generalized in \cite{GZ} to describe every multiplication map by a power of a linear form as a mixed order Hessian. These Hessians have become the main tool to detect Lefschetz properties for Artinian Gorenstein (AG) algebras, as soon as the Macaulay dual generator is known (see \cite{ AAISY2023, A2022, AI2020, BFP2022, DGI, Go, GZ,  GZ2,  HMV2020}). \\

When $\mathbb{K}$ is  an algebraically closed field of characteristic zero there is a differential version of the Macaulay-Matlis duality, closely related to the the classical apolarity, that has been proved to be very
efficient to understand the Lefschetz properties for the Gorenstein case. In particular, it introduces a  differential criteria to control the Lefschetz properties. When the Artinian algebra is not Gorenstein, very little is known about the Lefschetz properties and one of the main problems seems to be the lack of a differential criterion to detect them. \\

Here we consider a more general setting where the algebra is not necessarily Artinian, but it is of finite type $A=Q/I$. We will say that $A$ fails SLP if there is $s<k$ such that for every $L \in A_1$ the multiplication map $\times L^{k-s}:A_s \to A_k$ does not have maximal rank. The Macaulay-Matlis duality implies that the inverse system of an ideal is a finitely generated $Q$-module if and only if the algebra is Artinian (see \cite[Proposition 2.66]{HMMNWW}). On the other hand, in order to construct a geometric counterpart of this picture, we only need the graded pieces of the inverse system, which are all finitely dimensional $\K$-vector spaces. \\   

Artinian algebras of codimension two satisfy the SLP. If the codimension is at least three, there are examples of Artinian algebras failing the Lefschetz properties. The first example to be deeply understood was related to Togliatti surface, first studied by Enrico Togliatti in \cite{T2,T1}. It is a surface such that its second osculating space does not have the expected dimension. It turns out that this is related with the fact that the Macaulay-Matlis dual of its rational parametrization is a codimension three Artinian algebra failing the weak Lefschetz property from degree two to three. \\

The main result of \cite{MMRO2013} is a differential criterion to detect the failure of the Weak Lefschetz property for Artinian algebras whose ideal is generated by forms having the same degree $d$ and the failure occurs in the first possible degree by lack of injectivity. The differential criterion relates the failure of the WLP with the existence of non-trivial Laplace equations on the rational variety given by a projection of the $d$-th Veronese embedding. In other words, this variety has an osculating defect. The linear system defining the variety is the apolar of the linear system given by the graded piece of degree $d$ of the ideal. This result was generalized by \cite{GIV2014} to also include the lack of surjectivity. These results have strong consequences in the study of the so-called Togliatti systems (see Definition \ref{TogliattiSystem}). The algebraic and geometric properties of the Togliatti systems was extensively study in the past few years, see for instance \cite{AAM2019, CMMR2021, CMRS2022, GIV2014, MMR2016, MMRO2013,MMR, MMR2018}) for details.\\

Our main result is a differential criterion to detect the failure of SLP for any $\K$-algebra of finite type in any degree. From now on, $Q= \mathbb{K}[X_0, \cdots, X_n]$ is a polynomial ring with standard grading, $I \subset Q$ is a homogeneous ideal and $A=Q/I$.

\begin{theorem}
Let $I \subset Q$ be a homogeneous ideal generated by forms of arbitrary degree, $L\in Q_1$ a general linear form and $P=L^{\perp}$. Consider $$\varphi_k = \varphi_{I^{-1}_k}: \mathbb{P}^{n} \dashrightarrow \mathbb{P}^{h_k-1},$$ where $h_k=\dim A_k$ is the Hilbert function, the projection of the Veronese embedding by the linear system $|I^{-1}_k|$. Then the transpose of the matrix of the  multiplication map $\times L^{k-s}:A_s \to A_k$ is:
\[[\times L^{k-s}]^{tr} = (k-s)!\ \widetilde{\Jac}^{s}_{\varphi_k}(P).\]
Where $\widetilde{\Jac}^{s}_{\varphi_k}$ is the $k$-th Jacobian of
the map $\varphi_k$.
\end{theorem}

One of the main corollaries is the following.

\begin{corollary}
Let $d,k,s$ be positive integers, with $s<k$, $I\subset Q$ an  ideal generated by homogeneous polynomials of arbitrary degrees, $d$ the lowest degree of the polynomials in $I$, $A=Q/ I$ and $L\in A_1$ a general linear form. Suppose that $\dim A_{s} \leq \dim A_k$ and $s<d$, then the linear map $\xymatrix {A_{s}\ar[r]^{\times L^{k-s}}& A_k}$ is not injective with kernel of dimension $\delta$ if and only if the variety $X_k$, the closure of the image of $\varphi_k:\mathbb{P}^n \dashrightarrow \mathbb{P}^{h_k-1}$ the projection of the Veronese embedding by the linear system $|I^{-1}_k|$, satisfies $\delta$ {\em non trivial} Laplace equation of order $s$.
\end{corollary}

The main Theorem and the Corollary yield  all previously cited  foundational results, that is:

\begin{enumerate}
	\item[(i)] SLP for Artinian algebras of codimension two (see \cite[Prop. 4.4]{HMNW} and our Corollary \ref{cor:codimension2}).
	\item[(ii)] The higher Hessian criterion (see \cite[Theorem 3.1]{MW} and our Corollary \ref{cor:higherhess}).
	\item[(iii)] The mixed Hessian criterion (see \cite[Theorem 2.4]{GZ2} and our Theorem \ref{littlebread}).
	\item[(iv)] The relation between osculating defectiveness and failure of WLP, the so-called {\it Tea Theorem} (see \cite[Theorem 3.2]{MMRO2013} and our Corollary \ref{Tea}).
\end{enumerate}

In Section \ref{TS}, we present the first new application of our main result.  We show that the order of the Laplace equations detected by the Tea Theorem (cf. \cite{MMRO2013}) for minimal monomial Togliatti systems (see Definition \ref{TogliattiSystem}) is the smallest possible. We hope that such result could shed some light on the problem of classifying toric varieties based on their osculating and inflectional behavior, a task that was previously addressed by Perkinson in \cite{P2000}.\\

Finally, in Section \ref{GIH}, as another new consequence of our main result, we obtain a geometric interpretation of the degeneracy of the higher Hessian and mixed Hessian. The search for this kind of result has lasted a decade and it is part of an effort involving many mathematicians to obtain higher order analogous to classical algebraic differential results. It includes higher order Gauss map, higher order Jacobians, higher order Milnor algebras and higher order Gordan-Noether theory (see \cite{IdP, DGI, Ol, Za2,Za}). \\

\vspace{0.3cm}
{\bf Acknowledgements.} The first and second authors would like to thank the warm hospitality of Federal University of Pernambuco, where part of this work was done. This work was partially funded by FAPEMIG universal project number APQ-01655-22, and also PPGMAT-UFMG.  

\section{Notation and Preliminaries}
In this section we will recall some standard results and fix the notation for this work. $\K$ will denote an algebraically closed field of characteristic zero unless explicitly mentioned otherwise. By a graded $\mathbb{K}-$algebra we mean a graded $\mathbb{K}-$algebra of finite type.

\subsection{Macaulay-Matlis duality}
We start by recalling the notion and results of  the well known Macaulay-Matlis duality. The interested reader can see the details in \cite{HMMNWW}.

Let $V$ be an $(n+1)-$dimensional $\K-$vector space and set $ \displaystyle R= \bigoplus_{i\geq 0} \textrm{Sym}^iV^{\ast}$ and
$ \displaystyle Q= \bigoplus_{i\geq 0} \textrm{Sym}^iV$. Let $\{x_0, \cdots, x_n\},\ \{X_0, \cdots, X_n\}$ be the dual bases of $V^{\ast}$ and $V$ respectively. So, we have the identification $R=\K[x_0,\cdots , x_n]$ and $Q=\K[X_0,\cdots , X_n]$. There are products
\begin{equation*}
\begin{array}{ccc}
	\textrm{Sym}^iV\otimes \textrm{Sym}^jV^{\ast} & \to &\textrm{Sym}^{j-i}V^{\ast}\\
	\alpha \otimes f &\mapsto & \alpha(f)
\end{array}    
\end{equation*}
This action gives $R$ the structure of $Q-$module. Over a field of zero characteristic, one can think in $R$ as a polynomial ring and $Q$ as the associated ring of differential operators via the identification $X_i=\frac{\partial}{\partial x_i}$, where $X_i(x_j)=\delta_{ij}$. \\
Let $I\subset Q$ be a homogeneous ideal, we define the {\em Macaulay's inverse system} $I^{-1}$ for $I$ as
\begin{equation*}
	I^{-1} :=\{ f \in R ~| ~ \ \alpha(f)=0 \textrm{ for all }\alpha \in I\}.
\end{equation*}
Then $I^{-1}$ is a $Q-$submodule of $R$ that inherits a grading of $R$. Conversely, If $M\subset R$ is a graded $Q-$submodule, then $$\textrm{Ann}(M):=\{ \alpha \in Q~| ~ \ \alpha(f)=0\ \textrm{for all } f \in M\}$$ is a homogeneous ideal in $Q$.

There is a one-to-one correspondence given by Macaulay-Matlis duality:

$$\begin{array}{ccc}
\{\text{homogeneous ideals of} \ Q\} & \leftrightarrow & \{\text{graded}\ Q-\text{submodules}\ \text{of}\ R \}\\
\operatorname{Ann}(M) & \leftarrow & M\\
I & \rightarrow & I^{-1}
  \end{array}
$$

Under this correspondence, $I^{-1}$ is a finitely generated $Q$-module if and only if $Q/I$ is an Artinian algebra.
It is important to notice that even in the case where $I^{-1}$ is not finitely generated as a $Q$-module, the graded piece $(I^{-1})_k=I_k^{-1}$ is a finite dimensional vector space.

\iffalse
\begin{remark}
	In classical terminology, if $\alpha(f)=0$ and $\deg(f)=\deg(\alpha)$, then $f$ and $\alpha$ are said to be {\em apolar} to each other. In fact, the pairing
\begin{equation*}
	\begin{array}{rcl}
	Q_i\times R_i  &\to   & \K  \\
	(\alpha,f)&\mapsto 	&\alpha(f)
	\end{array}
\end{equation*}
is exact; it is called the {\em apolarity} action of $Q$ on $R$.

\end{remark}

\fi

For a given $f \in R_d$ define the Annihilator ideal
\begin{equation*}
\Ann_f = \{\alpha \in Q ~ | ~ \alpha(f)=0\}\subset Q.    
\end{equation*}

 We define $$A=\frac{Q}{\Ann_f}.$$
One can verify that $A$ is a standard graded Artinian Gorenstein $\K$-algebra such that $A_j=0$ for $j>d$ and such that $A_d \neq 0$ (see \cite[Section 1,2]{MW}).
We assume, without loss of generality, that $(\Ann_f)_1=0$.

Conversely, we get the following characterization of standard graded
Artinian Gorenstein $\K$-algebras.

\begin{theorem}{(Double annihilator Theorem of Macaulay)} \label{G=ANNF} \\
Let $R = \K[x_0,\cdots,x_n]$ and let $Q = \K[X_0,\cdots, X_n]$ be the associated ring of differential operators.
Let $A= \displaystyle \bigoplus_{i=0}^dA_i = Q/I$ be an Artinian standard graded $\mathbb K$-algebra. Then
$A$ is Gorenstein if and only if there exists $f\in R_d$
such that $A\simeq Q/\Ann_f$.
\end{theorem}

This result is well known and can be found in \cite[Theorem 2.1]{MW}.

\begin{definition} \rm With the previous notation, let $A= \displaystyle \bigoplus_{i=0}^dA_i = Q/I$ be an Artinian Gorenstein $\K$-algebra with $I = \Ann_f$, $I_1=0$ and $A_d \neq 0$. The socle degree of $A$ is $\text{Soc}(A) = d$ which coincides with the degree of the form $f$.
 The codimension of $A$ is the codimension of the ideal $I \subset Q$ which coincides with its embedding dimension, that is, $\codim A = n$.
\end{definition}

\subsection{Lefschetz properties}

For a graded $\K$-algebra of  finite type $\displaystyle A = \bigoplus_{k=0} A_k$, the Hilbert function of $A$ is an integer function defined by 

$$\operatorname{Hilb}(A)(k) = \dim A_k=h_k$$.

In the case $A$ is an Artinian algebra, its Hilbert function has only finitely many non zero entries and sometimes it is called Hilbert vector of the algebra.  
We say that a graded Artinian $\K$-algebra $\displaystyle A = \bigoplus_{k=0}^d A_k$ is standard graded if it is generated as algebra in degree $1$.

We now recall the so-called Lefschetz properties for standard graded Artinian $\K$-algebras.

\begin{definition}\rm With the previous notation, for $\alpha \in A_k$ we define the $\K$-linear multiplication map $\times \alpha:A_s \to A_{k+s}$.
\begin{enumerate}
 \item[(i)] We say that $A$ has the Strong Lefschetz property (SLP) if there is $L \in A_1$ such that the $\K$-linear multiplication maps $\times L^{s-k}:A_k \to A_{s}$ have maximal rank for all $0\leq k<  s$. In this case $L$ is called a Strong Lefschetz element.
 \item[(ii)] We say that $A$ has the Weak Lefschetz property (WLP) if there is $L \in A_1$ such that the $\K$-linear multiplication maps $\times L:A_k \to A_{k+1}$ are of maximal rank for $k\geq 0$. In this case $L$ is called a Weak Lefschetz element.
\end{enumerate}
 
\end{definition}

\begin{remark}\rm \label{SLPimpliesWLPimplesUnimodality}
 It is easy to verify that SLP implies WLP; it is also true that WLP implies the unimodality of $\operatorname{Hilb}(A)$. Both converses are not true (see  \cite[Theorem 3.8]{Go}).\\
In the case in which $A$ is a standard graded Artinian Gorenstein $\K$-algebra, it is enough to verify the SLP condition in complementary degrees, that is, $\times L^{d-2k}:A_k \to A_{d-k}$ for all $k\leq d/2$ (see \cite{MW}). It is called SLP in the narrow sense. Moreover, in this case one can verify the WLP in the middle to conclude WLP in general (see \cite[Rmk. 1.4]{GZ2}).
\end{remark}

\begin{definition}
	Let $A$ be a standard graded $\K$-algebra of finite type over $\K$. We say that $A$ fails the SLP if there is $s<k$ such that
	the multiplication map $\times L^{k-s}:A_s \to A_k$ does not have maximal rank for every $L \in A_1$.
\end{definition}

\subsection{Laplace equations}

Let $X \subset \mathbb{P}^N$ be a quasi-projective variety of dimension $n$ and $x \in X$ a smooth point. We can choose a system of local coordinates around $x$ and a local parametrization of $x \in  \mathcal{U}\subset X$ of the form
\[\varphi: \Delta \to \mathcal{U}; \] where $\Delta$ is a multidisc, $x = \varphi(0, \cdots, 0)$ and the components of $\varphi(t_1 , \cdots , t_n )$ are power series.
The tangent space to $X$ at $x$ is the $\K$-vector space generated by the $n$ vectors that are the partial derivatives of $\varphi$ at $x$. Since $x$ is a smooth point of $X$, these $n$  vectors are $\K$-linearly independent.
It is equivalent to say that the Jacobian matrix has maximal rank.
This is the tangent space in the sense of differential geometry (see \cite{GH}).

Similarly, the $s$-th osculating (vector) space $T_x^s X$ is the span of all partial derivatives of order $\leq s$ (see \cite{GH}).
The expected dimension of $T_x^s X$ is
\[\mathrm{exp.dim} ~T_x^s X = \binom{n+s}{s}-1.\]
In general we have \[\dim ~T_x^s X \leq  \binom{n+s}{s}-1.\]
We say that $X$ satisfies $\delta$ Laplace equations if for a generic smooth point $x \in X$ we have
\[\dim T_x^s X =  \binom{n+s}{s} -1 - \delta. \]
Sometimes we also say that $X$ has an osculating defect in order $s$.

By definition, the dimension $\dim T_x^s X =\rank Jac^s_{\varphi} -1$ where $J^s_{\varphi}$ is the matrix of $s-$jets of $\varphi$.    

Let $I \subset Q$ be a homogeneous ideal and let $I^{-1}\subset R$ be its Macaulay inverse system. Let $k$ be a positive integer and assume that  $\dim I_k = r>0$, with $I_k=<F_1, \cdots, F_r>$. Associated with the linear system $|I_k^{-1}|$ of dimension $h_k - 1 = {{n+k}\choose{k} }-r-1$ there is a rational map:

$$\varphi_k: \mathbb{P}^n\dasharrow \mathbb{P}^{h_k-1}.$$

The closure of its image $X_k = \overline{ \varphi_{(I^{-1})_k}(\mathbb{P}^n)}\subset \mathbb{P}^{h^k-1} $ is the projection of the $n$-dimensional Veronese variety $V(k,n)$ from the linear system \[I_k = \langle F_1,\cdots, F_r\rangle\subset | \mathcal{O}_{\mathbb{P}^{n}}(k))|.\]

\begin{definition}
Let $n,s,k$ be positive integers with $s< k$, and  $\mathcal{L} \subset | \mathcal{O}_{\mathbb{P}^{n}}(k)|$ be a linear system of projective dimension $\dim \mathcal{L} = h_k-1$.
Let $\varphi = \varphi_{\mathcal{L}}$, $\varphi:\mathbb{P}^n \dashrightarrow \mathbb{P}^{h_k-1}$ be the associated rational map, given by
$\varphi(P)=(F_1(P):\cdots:F_{h_k}(P))$. Let $Q=\K[X_0,\cdots,X_n]$ be the ring of differential operators and let $Q_k=<\alpha_1, \cdots, \alpha_m>$.  The $s$-th Jacobian matrix associated to the map is $\Jac^{s}_{\varphi} = (\alpha_i(F_j))_{m\times h_k}$.
\end{definition}
\iffalse
\begin{lemma} \label{lem} Let $\varphi:\mathbb{P}^n \dashrightarrow \mathbb{P}^{h_k-1}$ be a finite rational map and let $X$ be its image. Let $x \in X$ be a generic smooth point. Suppose that $\dim A_s < N$. Then
\[\dim T_x^s X ~=  \rank \Jac^s_{\varphi} -1 \leq \binom{n+s}{s}-1.\]
\end{lemma}

\begin{proof}
Since $x \in X$ is generic and $\varphi$ is finite, $x$ has a finite number of pre-images. Choose $p \in \mathbb{P}^n$ one of then and choose a small neighborhood $\mathcal{U}$ of $x$ such that $\mathcal{V} = \varphi^{-1}(\mathcal{U})$ avoid the others pre-images of $x$. The restriction of $\varphi:\mathcal{V} \to \mathcal{U}$ is a local parametrization whose components are the homogeneous forms of the linear system that induces $\varphi$. Since they are homogeneous forms, by Euler formula, their partial derivatives of order $i$ belong to the ideal generated by the derivatives in order $i+1$. Evaluating in a point we get that all derivatives of order $\leq s$ belong to the linear space of the derivatives of order $s$ in this point. The result follows.  
\end{proof}

\fi

\begin{remark}\label{remark:key}

Let $\varphi:\mathbb{P}^n \dashrightarrow \mathbb{P}^N$ be a rational map and let $X$ be its image. Let $x \in X$ be a generic smooth point. Suppose that $\dim A_s < N$. Then
\[\dim T_x^s X ~=  \rk~ \Jac^s_{\varphi} -1 \leq \binom{n+s}{s}-1.\]

When $X\subset \mathbb{P}^N$ satisfies $\delta$ Laplace equations of order $s$, then
\[\delta = \binom{n+s}{s}- \operatorname{rk} \Jac^s_{\varphi}. \]
\noindent We say that $X$ satisfies a Laplace equation of minimal order $s$, if $X$ satisfies one Laplace equation of order $s$ but does not satisfies any Laplace equations of order $k$ for $1 \leq k \leq s-1$.

Let $I \subset Q$ be a homogeneous ideal (not necessarily Artinian) and $I^{-1} $ its inverse system (not necessarily finitely generated as a $Q$-module).
Let $I^{-1}_k$ be the linear system defining $\varphi$, notice that it is finitely dimensional.
Let $c$ be the initial degree of $I$. If $s\geq c$ then we get some trivial Laplace equations for $X_d$ since we get $I_{s} \neq 0$, that is, there is $\alpha \in Q_{s}$ such that $\alpha(F_i) =0$ for all $i$.
\end{remark}

\begin{remark} \label{rmk:curves}\rm
Let $C \subset \mathbb{P}^N$ be a non-degenerate curve. If $s<N$ then $C$ does not satisfy a non trivial Laplace equation of order $s$. In fact, in \cite[1d]{GH} the authors show that for $s\leq N$, and $x \in X$ a general point we get
\begin{equation}\label{eq:curves}
\dim T^s_xX = s.    
\end{equation}

It is also trivial to show that for $s\geq N$, and $x \in X$ a general point we get \[ T^s_xX = \mathbb{P}^N.\]\\
In other words, the rank of $ \Jac^s_{\varphi}$ is always maximal.  
\end{remark}

\begin{example}\label{togliatti}
Let $I=(x^3,y^3,z^3,xyz)\in \mathbb{K}[x,y,z]$, consider its Macaulay inverse system given by $I^{-1}=(x^2y,x^2z, xy^2,xz^2, y^2z, yz^2)$. Associated with the linear system $|I_3^{-1}|$ we have the rational map:
$$\begin{array}{ccc}
 	\varphi:\mathbb{P}^2& \dasharrow &\mathbb{P}^5\\
 	(x:y:z)&\to &(x^2y:x^2z: xy^2:xz^2: y^2z: yz^2)
\end{array}$$
Let $X_3 = \overline{ \varphi(\mathbb{P}^2)}$. This surface is known as Togliatti Surface. It was first studied by Eugenio Togliatti, in \cite{T1} and \cite{T2}, where he proved that  osculating space do not have the expected dimension. Indeed,   at generic points $x\in X_3$, we have that

$$x^2\dfrac{\partial^2 \varphi}{\partial x^2}-xy\dfrac{\partial^2 \varphi}{\partial x \partial y}-xz\dfrac{\partial^2 \varphi}{\partial x \partial z}+y^2\dfrac{\partial^2 \varphi}{\partial y^2}-yz\dfrac{\partial^2 \varphi}{\partial y \partial z}+z^2\dfrac{\partial^2 \varphi}{\partial z^2}=0$$

%$$x^2\varphi_{x^2}-xy\varphi_{xy}-%xz\varphi_{xz}+y^2\varphi_{y^2}-%yz\varphi_{yz}+z^2\varphi_{z^2}=0, $$

\noindent is an order 2 Laplace equation for $X_3$.
\end{example}

As mentioned earlier, in \cite{MMR2016} the authors proved that the defectiveness of the osculating space in the case of  Example \ref{togliatti}, is related with the failure of WLP for the ideal $I=(x^3,y^3,z^3,xyz)\in \mathbb{K}[x,y,z]$. Indeed, for a generic linear form $L \in R_1$, it is easy to see that the map
$$\times L: (\mathbb{K}[x,y,z]/I)_2\to (\mathbb{K}[x,y,z]/I)_3$$
\noindent is neither injective nor surjective. In the next section, we will investigate this relationship further.

\section{The main result and its consequences}

In this section, we prove our main result that generalizes the main results of \cite{MMRO2013} (see Section \ref{TS}), \cite{GZ2} and  \cite{MW} (see Section \ref{GIH}). Moreover, as a consequence we reobtain SLP for codimension $2$ Artinian algebras first obtained in \cite{HMNW}.

\iffalse
In this section we prove our main result that generalizes the main result of \cite{GZ2}, \cite{MW} and \cite{MMRO2013}. Moreover, as a consequence we reobtain SLP for codimension $2$ Artinian algebras, see \cite{HMNW}.

In this section our aim is to see \cite{}the Laplace equations provided by the Tea Theorem (See [MMRO]) are indeed nontrivial. For instance, consider XXXXXX
\fi

Let $s,k$ be positive integers, with $s <k$. Let $I \subset Q$ be a homogeneous ideal generated by $r$ forms of arbitrary degree, $I^{-1}$ its Macaulay inverse system, $A=Q/I$, $L\in Q_1$ a general linear form and $P=L^{\perp}$ its dual. Consider $\varphi_k = \varphi_{I^{-1}_k}: \mathbb{P}^{n} \dashrightarrow \mathbb{P}^{h_k-1}$ the projection of the Veronese embedding by the linear system $|I^{-1}_k|$.

We define the {\em $s$th relevant Jacobian} of $\varphi_k$, to be the matrix $\widetilde{\Jac}^s_{\varphi_k}$ whose rows are given by a $\K$ linear basis of $A_s$. To be more precise, if $\varphi_k$ is defined by $\varphi_k=(F_1:\cdots:F_{h_k})$ and $A_s=<\gamma_1,\cdots,  \gamma_{h_s}>$, then $$\widetilde{\Jac}^s_{\varphi_k} = (\gamma_i(F_j))_{h_s \times h_k}.$$ 

For a point $P_0 \in \mathbb{P}^n$ we denote $\Jac^{s}_{\varphi_k}(P_0)$ the evaluation of the matrix $\Jac^{s}_{\varphi_k}$ in $P_0$.

Let $[\times L^{k-s}]$ be the matrix of the linear map $\times L^{k-s}: A_{s} \to A_k$ with respect to the ordered basis defining the Jacobian of order $s$ and the dual basis of $I^{-1}_k$.  

The following lemma sometimes called the differential Euler formula (see \cite[Chapter 7]{Ru}) will be needed in what follows.  

\begin{lemma}[Euler]\label{EulerLemma}
Let $d$ be positive integer, $F \in R_d$ and $L \in Q_1$ and $P=L^{\perp}$ then \[L^d\cdot F = d! \cdot F(P).\]
\end{lemma}

\begin{theorem}\label{cheesebread}
Let $I \subset Q$ be a homogeneous ideal such that $\dim I_k = r > 0$, $L\in Q_1$ a general linear form and $P=L^{\perp}$. Consider $$\varphi_k = \varphi_{I^{-1}_k}: \mathbb{P}^{n} \dashrightarrow \mathbb{P}^{h_k-1},$$ the projection of the Veronese embedding by the linear system $|I^{-1}_k|$. Then
\[  [\times L^{k-s}]^{tr} = (k-s)!\ \widetilde{\Jac}^{s}_{\varphi_k}(P).\]
\end{theorem}

\begin{proof} Let $\{ F_1, \cdots, F_{h_k} \} \subset R$ be an ordered basis for $I^{-1}_k$ and let $\{\gamma_1, \cdots, \gamma_{h_s}\}$ be an ordered basis for $A_{s}$. From Macaulay-Matlis duality we can choose a basis $\{\beta_1, \cdots, \beta_{h_k}\}$ for $A_{k}$, such that
\begin{equation}\label{eq:duality} \beta_i( F_j) = \delta_{i,j}.\end{equation}
\iffalse
Let $L = c_0X_0+ \cdots + c_nX_n$ be a linear form.
Moreover, for each $j = 1, \cdots, s$,  there exists $a_{ij} \in \K $, with $i = 1, \cdots, m$ such that
\begin{equation}
	L^{k-s}.\gamma_i^{s} = \sum_{i=1}^{m} a_{ji}\beta_j
\end{equation}
\fi
Moreover, for each $j = 1, \cdots, s$,  there exists $a_{ij} \in \K $, with $i = 1, \cdots, h_k$ such that
\begin{equation}
	L^{k-s}.\gamma_j = \sum_{i=1}^{h_k} a_{ij}\beta_i,
\end{equation}
that is, the matrix of the multiplication by $\times L^{k-s}$ is given by $[a_{ij}]_{(h_k)\times (h_s)}$.

On the other hand, note that the $(ji)-$ entry of $[\widetilde{\Jac}^{s}_{\varphi_k} ]$ is given by \linebreak $[\widetilde{\Jac}^{s}_{\varphi_k}]_{ji} = \gamma_j(F_i)$ and we have

\begin{equation}\label{proof}
\begin{array}{rcll}
	\gamma_j(F_i(P)) & =& \dfrac{1}{(k-s)!}\ L^{k-s}\gamma_j\cdot F_i&\textrm{from Lemma \ref{EulerLemma} }\\
 	& =&\displaystyle \dfrac{1}{(k-s)!}\ \left(\sum_{i=1}^{m}a_{ij}\beta_i\right)\cdot F_i& \textrm{from Equation \eqref{eq:duality}}\\
 	&=&\dfrac{1}{(k-s)!}\ a_{i,j}&
\end{array}
\end{equation}

\noindent Using equation \eqref{proof} for every $i = 1, \cdots, h_k$ and every $j = 1, \cdots, h_s$, the result follows.

\end{proof}

\begin{remark}\label{remark:rank}
	Let $\{\gamma_1,\cdots, \gamma_{h_s},\alpha_1,\cdots, \alpha_t\}$ be an ordered basis for $Q_{s}$, note that the $(ij)-$entry of $[\Jac^{s}_{\varphi_k} ]$ is given by $[\Jac^{s}_{\varphi_k}]_{ij} = \eta_j( F_i$), where \linebreak $\eta_j \in \{\gamma_1,\cdots, \gamma_{S},\alpha_1,\cdots, \alpha_t\}$. If $\eta_j \in \{\gamma_1,\cdots, \gamma_{h_s}\}$ we have $[\Jac^{s}(\varphi_{k})]_{ij} = \gamma_j(F_i)$ but if $\eta_j \in \{\alpha_1,\cdots, \alpha_t\}$ we have $\eta_j( F_i) =0$, that is, the last $t$ rows of the matrix  $[\Jac^{s}(\varphi_{k})]$ are zero. And we have

   \[ [\Jac^{s}(\varphi_{k})]=  \left[\begin{array}{c}
     	[\widetilde{\Jac}^{s}(\varphi_{k})] \\
      	0
	\end{array}\right]\]

	Therefore   \[\rk (\widetilde{\Jac}^s_{\varphi_k}) = \rk (\Jac^{s}_{\varphi_k}).\]
\end{remark}

\begin{corollary}\label{cor:codimension2}
Let $A$ be a standard graded Artinian $\K$-algebra of codimension two, then $A$ has SLP.
\end{corollary}

\begin{proof} Since A is a  height two Artinian algebra, we can consider $A = Q/I$, where $Q=\mathbb{K}[X_0,X_1]$ and $I\subset Q$  an Artinian ideal generated by homogeneous polynomials of arbitrary degree. Let $k,s$ be positive integers, with $s<k$ and $L \in A_1$ a general linear form and consider the maps $\times L^{k-s}:A_s \to A_{k}$ and $\varphi_k:\mathbb{P}^1 \dashrightarrow C \subset \mathbb{P}^N$ the projection of the Veronese embedding by the linear system $|I^{-1}_k|$.

From Theorem \ref{cheesebread}, the rank of the map $\times L^{k-s}$ is the rank of the Jacobian $\widetilde{\Jac}^s_{\varphi_k}$ which is always maximal due to Remark \ref{rmk:curves}.  % from \cite[1d]{GH},from Equation \eqref{eq:curves} on  

\end{proof}

\iffalse
\begin{corollary}
Let $d,k,s$ be positive integers, with $s<k$, $I\subset Q$ an Artinian ideal generated by homogeneous polynomials of arbitrary degree, $A=Q/ I$ and $L\in A_1$ a general linear form. Then $A$ fails to have maximal rank from $\xymatrix {A_{d-k}\ar[r]^{\times L^k}& A_d}$ if, and only if, the variety $X_d$, the projection of the Veronese embedding by the linear system $|I^{-1}_k|$, satisfies at least one {\em non trivial} Laplace equation of order $d-k$.
\end{corollary}
\fi
\begin{corollary}\label{gentea}
Let $d,k,s$ be positive integers, with $s<k$, $I\subset Q$ an ideal generated by homogeneous polynomials of arbitrary degrees, $d$ the lowest degree of the polynomials in $I$, $A=Q/ I$ and $L\in A_1$ a general linear form. Suppose that $\dim A_{s} \leq \dim A_k$ and $s<d$, then the linear map $\xymatrix {A_{s}\ar[r]^{\times L^{k-s}}& A_k}$ is not injective with kernel of dimension $\delta$ if and only if the variety $X_k$, the image of $\varphi_k:\mathbb{P}^n \dashrightarrow \mathbb{P}^{h_k-1}$ the projection of the Veronese embedding by the linear system $|I^{-1}_k|$, satisfies $\delta$ {\em non trivial} Laplace equation of order $s$.
\end{corollary}

\begin{proof}
First note that, since $s<d$ one has
\begin{equation}\label{dimAs}
	\dim A_{s}= \binom{n+s}{s}.
\end{equation}

Therefore
\begin{equation}\label{delta}
\begin{array}{rcll}
	\delta & =& \dim \ker(\times L^{k-s})&\\
 & =& \dim A_s - \rk [\times L^{k-s}]& \\
	& =& \dim A_s - \rk \widetilde{\Jac}^s_{\varphi_k}&\textrm{from Theorem \ref{cheesebread}} \\
 	&=& \binom{n+s}{s} - \rk \widetilde{\Jac}^s_{\varphi_k}& \textrm{from Equation \eqref{dimAs}}\\
 	&=& \binom{n+s}{s}- \rk \Jac^s_{\varphi_k}& \textrm{from Remark \ref{remark:rank}}
\end{array}
\end{equation}

The result  follows from Remark \ref{remark:key}.
\end{proof}

Note that when $I\subset Q$ is an ideal generated by $r$ homogeneous polynomials $\{F_1, \cdots, F_r\}$ of degree $d$, $A_s=({Q/I})_s=Q_s$, then by Remark \ref{remark:rank} $\widetilde{\Jac}^{s}_{\varphi_d} =\Jac^{s}_{\varphi_d} $ for every positive integer $s$. By this observation, considering $k=d$ and $s=d-1$ on Theorem \ref{cheesebread} and Corollary \ref{gentea}, we can reobtain the celebrated Tea Theorem see \cite[Theorem 3.2]{MMRO2013}.

\begin{corollary}[The Tea-Theorem]\label{Tea}
Let $I \subset Q$ be an Artinian ideal generated by $r$ homogeneous polynomials $\{F_1, \cdots, F_r\}$ of degree $d$, with $r \leq {{n+d-1}\choose{n-1}}$.  Then the following conditions are equivalent:
\begin{itemize}
	\item[i)] The ideal $I$ fails the WLP in degree $d-1$;
	\item[ii)] The variety $X_d$ satisfies at least one Laplace equation of order $d-1$.
\end{itemize}
\noindent
\end{corollary}

This theorem explains, for instance, the relation between the failure of WLP and the osculating defectiveness of the Togliatti surface in example \ref{togliatti}. Due to this,  we recall the following definitions, proposed in \cite[Definition 3.5]{MMRO2013}.

\begin{definition}\normalfont \label{TogliattiSystem}Let $I \subset Q$ be an Artinian ideal generated by $r$ forms  of degree $d$, and $r \leq {n+d-1\choose n-1}$. We will say that:
\begin{itemize}
\item[(i)] $I$ is a {Togliatti system} if it satisfies one of two equivalent conditions in Corollary \ref{Tea}.
\item[(ii)] $I$ is a {monomial Togliatti system} if, in addition, $I$ can be generated
by monomials.
\item[(iii)] $I$ is a {smooth Togliatti system} if, in addition, the rational variety $X_d$ is smooth.
\item[(iv)] A monomial Togliatti system $I$ is {minimal} if  there is no proper subset of the set of generators defining a monomial Togliatti system.
\item[(v)] The number of generators of a Togliatti system will be denoted by $\mu(I)$.

\end{itemize}
\end{definition}

We highlight that the condition on the numbers of generators of the ideal $I$ ensures that the Laplace equations found by the Tea Theorem are not trivial in the sense of Remark \ref{remark:key}.

It also should be noted that in the original statement of the Tea Theorem, there was a third equivalent condition on the linearity of the forms $\{F_1, \cdots, F_r\}$ restricted to a general hyperplane of $\mathbb{P}^{n}$. It was omitted here since it does not relate the Lefschetz properties of $I$ with the osculatory behavior of $X_d$,.

The Tea Theorem poses two important questions: Are the Laplace equations found of minimal order? Exactly how many Laplace equations do we have? We will concentrate on these two questions in the next section.

\section{Monomial Togliatti Systems}\label{TS}

In this section, we will prove that, for minimal monomial Togliatti systems, the Laplace equations found by Corollary \ref{gentea} and Corollary \ref{Tea}, have minimal order. Such a result opens the possibility of classifying higher dimensional toric varieties by its osculating behavior, in the spirit of \cite{BPT1992}, \cite{FKPT1985} and \cite{P2000}.

\iffalse

\begin{example}
	Consider the ideal

\[
\begin{array}{c}
I=(x_1^2, x_2^2, x_3^2, x_4^2, u_1^2, u_2^2, u_3^2, u_4^2, u_2u_4, x_2u_4, x_1u_4, u_1u_3, x_4u_3, x_2u_3, x_1u_3, x_4u_2, x_3u_2, \\
 x_1u_2,x_4u_1,x_3u_1, x_2u_1, x_1u_1-x_2u_2+x_3u_3+x_3u_4-x_4u_4 x_3x_4, x_2x_4, x_1x_4, x_2x_3, x_1x_3, x_1x_2)
\end{array}
\]
    
   in the polinomial ring $R=\mathbb{K}[x_1, x_2, x_3, x_4, u_1, u_2, u_3, u_4]$, it has $28=\displaystyle {{8}\choose{6}}$ minimal generators. Its inverse Macaulay system is $$I^{-1} = (u_3u_4, u_1u_4, x_3u_4+x_4u_4, u_2u_3, x_3u_3+x_4u_4, u_1u_2, x_2u_2-x_4u_4, x_1u_1+x_4u_4)$$.  
\end{example}

Note that in the above example, the ideal $I$ is the range $r \leq {{n+d-1}\choose{n-1}}$, originally stated by the authors. But as the Theorem \ref{cheesebread} shows, the existence of a Laplace equation is enough to find examples of ideals failing the Lefschetz properties, is indenpendent of this condition.  XXXX incluir ideal de grau 3.

Next, we show that the condition $r \leq {{n+d-1}\choose{n-1}}$, on the numbers of generators of the ideal $I = (F_1, \cdots, F_r)\subset R$ is enough to ensure that the Laplace equations found by the Tea Theorem are non trivial.

 Let $I$ be an Artinian ideal generated by monomials of degree $d$, in this case its associated variety $X_d$ is toric, hence it can be studied using combinatorial methods, as one can see in \cite{AAM2019}, \cite{MMR2016} and in what follows.
\fi
\iftrue

 Let $I$ be an Artinian ideal generated by monomials of degree $d$. In this case its associated variety $X_d$ is toric, hence it can be studied using combinatorial methods, as one can see in \cite{AAM2019}, \cite{MMR2016} and in what follows. Let $I^{-1}$ be its Macaulay inverse system and $A_I$ be the set of integral points corresponding to monomials in $I^{-1}$. We denote by $\Delta_{n}$ the standard $n$-dimensional simplex in the lattice $\mathbb{Z}^{n+1}$. We consider $d \Delta_n$ and define the polytope $P_I$ as the convex hull of the finite subset $A_I \subset \mathbb{Z}^{n+1}$. Additionally, we denote by $T^{s}X_d$ the $s-$osculating space of $X$ at a generic smooth point $P \in X_d$.

\begin{proposition}\label{polytope}
Let $I \subset Q $ be an Artinian ideal generated by $r$ monomials of degree $d$. Assume $r \leq {n+d-1 \choose n-1}$.
\begin{enumerate}
	\item[a)] Then $s\in \{1,\cdots, d-1\}$ is the least integer such that $X_d$ satisfies a Laplace equation of order $s$  if and only if there exists a hypersurface $F_{s}$ of degree $s$ containing $A_I \subset \mathbb{Z}^{n+1}$.
	\item[b)] Moreover, $I$ is a minimal Togliatti system if and only if any such hypersurface does not contain any integral point of $d \Delta_n \setminus A_I$ except possibly some of the vertices of $d \Delta_n$.

\end{enumerate}
\end{proposition}
\begin{proof}
%See \cite[Proposition 3.4]{MMR2016}.

  Let $I$ be as in the statement. By \cite[Proposition 2.2]{MMRN2011} to study the presence of WLP and SLP in $Q/I$, it is enough to verify if $L=X_0+ \cdots +X_n$ is a Lefschetz element for $Q/I$. Hence by \ref{cheesebread}, in order to study Laplace equations on $X_d$, that is, the dimension of $T^s_PX_d$ at a generic point $P \in X_d$, it is enough to consider $P=(1,\cdots, 1)$.
Thus, $T^s_PX_d$ does not have the expected dimension if and only if $\Jac^s_{\varphi}(P)$ does not have complete rank. By \cite[Proposition 1.1]{P2000} this is equivalent to say that the dimension of the linear space of polynomials in $n+ 1$ variables and degree $\leq s$ that are satisfied by the lattice points in $A_I$ is positive, i.e. there exists a hypersurface $F_s$ of degree at most $s$ containing $A_I$. We now proceed to prove that the degree of $F_s$ is $s$. Suppose, by contradiction, that the degree of $F_s$ is $k<s$, again by  \cite[Proposition 1.1]{P2000} we will have that $T^k_PX_d$ does not have the expected dimension, which contradicts the minimality of $s$, hence $\deg F_s = s$.
\\
 For the second part, see \cite[Proposition 3.4]{MMR2016}

\end{proof}
For $s=d-1$, the above proposition is the same as \cite[Proposition 3.4]{MMR2016}. {\color{black}Note that, in that case, if $X_d$ satisfies one Laplace equations of order $s$ smaller than $d-1$, then by the above Lemma, there is a hypersurface $F_s$, of degree $s$ containing $A_I \subset \mathbb{Z}^{n+1}$. Then $X_0^{d-1-s}\cdot F_s$ is a hypersurface of degree $d-1$ containing $A_I$}.

By Corollary \ref{Tea}, the $(d-1)-$osculating space of $X_d$ does not have the expected dimension. Next we will use Theorem \ref{cheesebread} and Proposition \ref{polytope} to investigate further the osculating behavior of $X_d$.

\begin{theorem}\label{order}
Let $I \subset Q$ be a minimal monomial Togliatti system of degree $d$ with $r$ generators. For $d \geq 3$, $X_d$ does  not satisfy Laplace equations of order $s$ for $1 \leq s \leq d-2$.
\end{theorem}
\begin{proof}
Let $I$ and $X_d$, be as above. Suppose that $X_d$ satisfies at least one Laplace equation of order $s<d-1$. We can assume without loss of generality that $s=d-2$. By Proposition \ref{polytope} there exists a degree $d-2$ hypersurface $F_{d-2}$, containing $A_I$. Let $H$ be the equation of a hyperplane that passes through any point of $d \Delta_n \setminus A_I$. Clearly $F_{d-1} = H\cdot F_{d-2}$ is a degree $d-1$ hypersurface that passes through all points of $A_I$, includes points of  $d \Delta_n \setminus A_I$, which contradicts the minimality of $I$. Hence $s=d-1$ as claimed.

\iffalse Since $\dim A_{s} = \binom{n+s}{s}$ for $1 \leq s \leq d-2$ and $\dim A_d = \binom{n+d}{d}-r$, we need to prove that $\times L^{d-s}: A_{s} \to A_d$ is injective for all $ 2 \leq s \leq d-2$.
\fi
\end{proof}
\color{black}{}
It is interesting to note, that a main difference between the Tea Theorem (\cite[Theorem 3.2]{MMR2016}), Theorem \cite[Theorem 5.1]{GIV2014} and our Theorem \ref{cheesebread}, is that the later allows to understand not only the Laplace equations of a projective variety but also its inflection points (see Section \cite[Section 1]{P2000}). Let $\varphi:\mathbb{P}^n \dashrightarrow \mathbb{P}^N$ be a rational map and let $X$ be its image.  We will say that $X$ has an inflection point of order $s$ at $P\in X$ if
\[\dim T_P^s X ~=  \rk \Jac^s_{\varphi} -1 < \binom{n+s}{s}-1.\]
\noindent Note that $X$, obviously, could have inflection points of order $s$ even if it does not satisfy Laplace equations of order $s$, as the following example shows.

\begin{example}
Let $n,d \geq 4$ and consider the Togliatti system $I = (x_0^d,\cdots,x_n^d)+x_0^{d-2}x_1(x_0,\cdots,x_n)$, and the algebra $A = \mathbb{K}[x_0,\cdots,x_n]/I$. By \cite[Proposition 3.19]{MMR2016} one has that the associated toric variety $X_d$ is not smooth. Hence there is a point $P \in X_d$ such that $\Jac^1_{\varphi_{d}}(P)$ is not a full rank matrix, and the same holds for $\times[P^{\perp}]^{d-1}: A_{1} \to A_d$.
\end{example}

More interesting is the fact that in the context of  Theorem \ref{cheesebread} we can use the geometrical properties of the variety $X$ associated with the rational map given by the linear system $|I^{-1}|_k$ to determine
whether the multiplication by a power of an arbitrary linear form $[\times L^{k-s}]: A_s\to A_k$ has full rank, as we show in the following example.

\begin{example}
Let $n,d \geq 4$ and consider the Togliatti system $I = (x_0^d,\cdots,x_n^d)+x_0^{d-1}(x_0,\cdots,x_n)$, and the algebra $A = \mathbb{K}[x_0,\cdots,x_n]/I$. By \cite[Theorem 2.6]{AAM2019} one has that the associated toric variety $X_d$ is smooth. Hence for every  point $P \in X_d$, one has that $\Jac^1_{\varphi_{d}}(P)$ is a full rank matrix, so it is $\times[P^{\perp}]^{d-1}: A_{1} \to A_d$. This means that for every linear form $L \in A_1$, the multiplication map $\times[L]^{d-1}: A_{1} \to A_d$ has full rank.
\end{example}

Now, we move towards the problem of computing the exact number of Laplace equations that a minimal monomial Togliatti system satisfies.

\begin{proposition}
	Let $n,d \geq 4$ and consider the Togliatti system $I = (x_0^d,\cdots,x_n^d)+x_0^{d-1}(x_0,\cdots,x_n)$. Then the associated variety $X_d$ satisfies exactly one Laplace equation of order $d-1$.
\end{proposition}
\begin{proof}
	By Theorem \ref{order}, $X_d$ will only have Laplace equations of order $d-1$. By
	\cite[Proposition 2.2]{MMRN2011}, we only need to study the osculating behaviour of $X_d$ at $P=L^{\perp}$, where $L = x_0+\cdots+x_n$, and by Corollary \ref{gentea}, the number of Laplace equations that $X_d$ satisfies is equal to the dimension of the kernel of $\times[L]: A_{d-1} \to A_d$.
    
	Let $f = \displaystyle \sum_{i \vdash (d-1)} a_{i}x^i$, where $i \vdash (d-1)$ means that $i=(i_0,\cdots, i_n)$ is a partition of $d-1$, and $x^i$ means $x_0^{i_0}x_1^{i_1}\cdots x_n^{i_n}$. Clearly, $f \in A_{d-1}$. Suppose that $f \in \ker (\times[L])$ and assume that $f\neq 0$.%, then $(x_0+\cdots+x_n)\cdot f \in I$.
    
	For each $i=(i_0,\cdots,i_n)\vdash d$, let $J_i$  be the set of $(n+1)-$uples of the form $(j_0,\cdots,j_n)$ such that for each $(j_0,\cdots,j_n)$ there exists only one $k$ such that $j_k+1=i_k$, with $j_k\geq 0$,  and $j_s=i_s$ for all other values of $s=1,\cdots,\hat{k},\cdots,n$. Note that
    
	$$(x_0+\cdots+x_n)\cdot f = \displaystyle \sum_{i \vdash d}(\sum_{j \in J_i}a_j)x^i$$

Note, for instance, that using the above notation, we see that the coefficient of
$x_1^{d-1}x_2$ is $(a_{(0,d-2,1,\cdots,0)}+a_{(0,d-1,0,\cdots,0)}$ and the coefficient of $x_1^{d-2}x_2x_3$ is $(a_{(0,d-2,1,\cdots,0)}+a_{(0,d-3,0,1,\cdots,0)}+a_{(0,d-3,1,1,\cdots,0)})$.

From the condition $(x_0+\cdots+x_n)\cdot f \in I$, and the fact that  $f \in \ker (\times[L])$, we get the following linear system:

\begin{equation}\label{linearalgebra}
	(\sum_{j \in J_i}a_j)=0,
\end{equation}
\noindent for every $i \vdash d$, except for $i=(d,0,\cdots,0)$, since $I = (x_0^d,\cdots,x_n^d)+x_0^{d-1}(x_0,\cdots,x_n)$, hence the term $(x_0+\cdots+x_n)\cdot a_{(d-1,0,\cdots,0)}x_0^{d-1}$ will lie in $I$ making $a_{(d-1,0,\cdots,0)}$ that corresponds to the coefficient of $x_0^{d-1}$ in $f$ to be free.

Solving the linear system \ref{linearalgebra} we get that $a_j = 0$ for every $j \in J_i$ and every $i \vdash d$ that give us that $\ker (\times[L]) =\langle x_0^{d-1} \rangle$. By the Corollary \ref{gentea} we have that $X_d$ satisfies precisely one Laplace equation.
\end{proof}
Using Macaulay2, we computed the dimension of the kernel resulting from the multiplication by $(x_0+\cdots+x_n)$ in explicit examples of minimal monomial Togliatti systems as described in \cite[Section 4]{AAM2019} and \cite[Section 4]{MMR2016}. Remarkably, we observed that these systems satisfy precisely  one Laplace equation of order $d-1$. Although the proof presented here heavily relies on the specific form of the Togliatti system, establishing this fact seems beyond the scope of our current techniques since a general classification for monomial Togliatti systems is lacking. Nevertheless, the significant number of examples in which the claim holds leads us to formulate the following conjecture.

\begin{conjecture}
Every minimal monomial Togliatti system of degree $d$ satisfies precisely one Laplace equation of order $d-1$.
\end{conjecture}
\color{black}{}

\section{Geometric Interpretation of Hessians}\label{GIH}

In this section we will study the geometric meaning of mixed and higher order hessians of a form $f$ of degree $d$ (see \cite{DGI}, \cite{GZ2}, \cite{MW} for more details).

\begin{definition}
 The $k$-th order Jacobian ideal of $f \in R$ is
\[J^k =J^k(f)= (Q_k \cdot  f) = (A_k \cdot f).\] It is the ideal generated by the $k$-th order partial derivatives of $f$.

The $k$-th polar mapping (or $k$-th gradient mapping) of the hypersurface $X=V(f) \subset \mathbb{P}^n$ is the rational map  $\Phi^k_X:  \mathbb{P}^n  \dashrightarrow   \mathbb{P}^{h_k-1}$ given by the linear system of $k$-th order partial derivatives of $f$ in $A_k$. The $k$-th polar image of $X$ is  $X_k = \overline{\Phi^k_X(\mathbb{P}^n)}\subseteq \mathbb{P}^{\binom{n+k}{k}-1}$, the closure of the image of the $k$-th polar map.
\end{definition}

\begin{lemma} \label{easy} Let $A = Q/I$ be an AG algebra with $I =\mathrm{Ann}~f$, $f \in R_d$ and $e<d$. Then
\[I^{-1}_k = J_k^{d-k}.\]

\end{lemma}

\begin{proof} Let $A_{d-k} = <\alpha_1, \cdots, \alpha_s>$, then
$J^{d-k} =(\alpha_1(f), \cdots, \alpha_s(f))$. Therefore,
$J^{d-k}_k =<\alpha_1(f), \cdots, \alpha_s(f)>\subset R_k$. By Poincaré duality, let $A_d = <\theta>$ and let $A_k = <\alpha_1^*, \cdots, \alpha_s^*>$  such that $\alpha_i\alpha_j^*=\delta_{ij}\theta$. Consider the lifting of $\alpha^*_i$ to $Q$ and complete  basis to $Q_{k} = <\alpha^*_1, \cdots, \alpha^*_s, \beta_1, \cdots, \beta_l>$ with
$I_{k} = <\beta_1, \cdots, \beta_l>$. Let $\alpha_i(f) \in J_k^{d-k}$, then we get $\alpha_i(f) \in I_k^{-1}$. Indeed, since $\beta_j(f)=0$ we get \[\beta_j(\alpha_i(f)) = \alpha_i(\beta_j(f))=\alpha_i(0)=0. \]
Therefore, $J_k^{d-k} \subset I_k^{-1}$ and, since they have the same dimension, the result follows.
\end{proof}

Theorem \ref{cheesebread} allows us to study the rank of the multiplication map in terms of the projection of the Veronese variety of the linear system given by the Macaulay inverse of an ideal $I$ not necessarily generated by forms of the same degree, we have the following result. The first equivalence is just a reinterpretation of  \cite[Theorem 2.4]{GZ2}.

\begin{theorem}\label{littlebread}
Let $A = Q/I$ be a standard graded Artinian Gorenstein $\K$-algebra such that $I = \Ann_f $ with $f \in R_d$. Let $0<s<k<d$ be integers and suppose that $\dim A_s \leq \dim A_k$. Let $L \in A_1$ be a linear form and $P=L^{\perp}\in \mathbb{P}^n$ be the corresponding point.
Then the following assertions are equivalent.   
\begin{enumerate}
	\item The $\K$-linear map $\times L^{k-s}: A_s \to A_k $ is not injective;  
	\item The mixed Hessian $\Hess_f^{s,d-k}(P)$ is not of maximal rank;
	\item The variety $X_k$ has osculating defect in order $s$.
\end{enumerate}
\end{theorem}

\begin{proof} By Theorem \ref{cheesebread}, we get  \[[\times L^{k-s}] = \widetilde{\Jac}^{s}_{\varphi_k}\]

On the other hand, by Lemma \ref{easy} \[\Hess_f^{(s,d-k)}(P) = \widetilde{\Jac}^{s}_{\varphi_k}(P).\]

The result follows.
\end{proof}

\begin{corollary}\label{cor:higherhess}
Let $A = Q/I$ an Artinian Gorenstein algebra with $I=\mathrm{Ann}_f$ of codimension $n+1$ and socle degree $d$. Then, the following assertions are equivalent.
\begin{enumerate}
	\item[(i)] $A$ has $SLP$;
	\item[(ii)] For every $k \leq \frac{d}{2}$, we get $\hess^k_f \neq 0$;
	\item[(iii)] for every  positive integer $k<\frac{d}{2}$,  $T^{k}X_{d-k}$ has the expected dimension.
\end{enumerate}
\end{corollary}

\begin{proof}
Recall that for an AG algebra the SLP is equivalent to the SLP in the narrow sense. The result now follows from  Theorem \ref{littlebread}.
\end{proof}

\bibliography{mref}

\begin{thebibliography}{10}

\bibitem{AAISY2023}
N.~Abdallah, N.~Altafi, A.~Iarrobino, A.~Seceleanu, and J.~Yam{\'e}ogo.
\newblock Lefschetz properties of some codimension three artinian gorenstein
  algebras.
\newblock {\em Journal of Algebra}, 625:28--45, 2023.

\bibitem{AAM2019}
C.~Almeida, A.~V. Andrade, and R.~M. Miró-Roig.
\newblock Gaps in the number of generators of monomial togliatti systems.
\newblock {\em Journal of Pure and Applied Algebra}, 223(4):1817 -- 1831, 2019.

\bibitem{A2022}
N.~Altafi.
\newblock Hilbert functions of artinian gorenstein algebras with the strong
  lefschetz property.
\newblock {\em Proceedings of the American Mathematical Society},
  150(2):499--513, 2022.

\bibitem{AI2020}
N.~Altafi, A.~Iarrobino, and L.~Khatami.
\newblock Complete intersection jordan types in height two.
\newblock {\em Journal of Algebra}, 557:224--277, 2020.

\bibitem{BPT1992}
E.~Ballico, R.~Piene, and H.-S. Tai.
\newblock A characterization of balanced rational normal surface scrolls in
  terms of their osculating spaces. {II}.
\newblock {\em Math. Scand.}, 70(2):204--206, 1992.

\bibitem{Be}
N.~Bergeron.
\newblock Lefschetz properties for arithmetic real and complex hyperbolic
  manifolds.
\newblock {\em Int. Math. Res. Not.}, (20):1089--1122, 2003.

\bibitem{BFP2022}
D.~Bricalli, F.~F. Favale, and G.~P. Pirola.
\newblock A theorem of gordan and noether via gorenstein rings.
\newblock {\em arXiv preprint arXiv:2201.07550}, 2022.

\bibitem{CMMR2021}
L.~Colarte-G{\'o}mez, E.~Mezzetti, and R.~M. Mir{\'o}-Roig.
\newblock On the arithmetic cohen--macaulayness of varieties parameterized by
  togliatti systems.
\newblock {\em Annali di Matematica Pura ed Applicata (1923-)}, 200:1757--1780,
  2021.

\bibitem{CMRS2022}
L.~Colarte-G{\'o}mez, E.~Mezzetti, R.~M. Mir{\'o}-Roig, and M.~Salat-Molt{\'o}.
\newblock Togliatti systems associated to the dihedral group and the weak
  lefschetz property.
\newblock {\em Israel Journal of Mathematics}, 247(1):195--215, 2022.

\bibitem{IdP}
P.~De~Poi and G.~Ilardi.
\newblock On higher gauss maps.
\newblock {\em Journal of Pure and Applied Algebra}, 219(11):5137--5148, 2015.

\bibitem{GIV2014}
R.~Di~Gennaro, G.~Ilardi, and J.~Vall\`es.
\newblock Singular hypersurfaces characterizing the {L}efschetz properties.
\newblock {\em J. Lond. Math. Soc. (2)}, 89(1):194--212, 2014.

\bibitem{DGI}
A.~Dimca, R.~Gondim, and G.~Ilardi.
\newblock Higher order jacobians, hessians and milnor algebras.
\newblock {\em Collectanea Mathematica}, 71(3):407--425, 2020.

\bibitem{FKPT1985}
W.~Fulton, S.~Kleiman, R.~Piene, and H.~Tai.
\newblock Some intrinsic and extrinsic characterizations of the projective
  space.
\newblock {\em Bull. Soc. Math. France}, 113(2):205--210, 1985.

\bibitem{Go}
R.~Gondim.
\newblock On higher hessians and the lefschetz properties.
\newblock {\em Journal of Algebra}, 489:241--263, 2017.

\bibitem{GZ}
R.~Gondim and G.~Zappal\`a.
\newblock Lefschetz properties for {A}rtinian {G}orenstein algebras presented
  by quadrics.
\newblock {\em Proc. Amer. Math. Soc.}, 146(3):993--1003, 2018.

\bibitem{GZ2}
R.~Gondim and G.~Zappal\`a.
\newblock On mixed {H}essians and the {L}efschetz properties.
\newblock {\em J. Pure Appl. Algebra}, 223(10):4268--4282, 2019.

\bibitem{GH}
P.~Griffiths and J.~Harris.
\newblock Algebraic geometry and local differential geometry.
\newblock {\em Ann. Sci. \'{E}cole Norm. Sup. (4)}, 12(3):355--452, 1979.

\bibitem{HMMNWW}
T.~Harima, T.~Maeno, H.~Morita, Y.~Numata, A.~Wachi, and J.~Watanabe.
\newblock {\em The Lefschetz Properties}.
\newblock Lecture Notes in Mathematics. Springer Berlin Heidelberg, 2013.

\bibitem{HMNW}
T.~Harima, J.~C. Migliore, U.~Nagel, and J.~Watanabe.
\newblock The weak and strong {L}efschetz properties for {A}rtinian
  {$K$}-algebras.
\newblock {\em J. Algebra}, 262(1):99--126, 2003.

\bibitem{HL}
M.~Harris and J.-S. Li.
\newblock A {L}efschetz property for subvarieties of {S}himura varieties.
\newblock {\em J. Algebraic Geom.}, 7(1):77--122, 1998.

\bibitem{HMV2020}
H.~Huang, M.~Micha{\l}ek, and E.~Ventura.
\newblock Vanishing hessian, wild forms and their border vsp.
\newblock {\em Mathematische Annalen}, 378(3-4):1505--1532, 2020.

\bibitem{Ka}
H.~Kasuya.
\newblock Minimal models, formality, and hard {L}efschetz properties of
  solvmanifolds with local systems.
\newblock {\em J. Differential Geom.}, 93(2):269--297, 2013.

\bibitem{KN}
M.~Kubitzke and E.~Nevo.
\newblock The {L}efschetz property for barycentric subdivisions of shellable
  complexes.
\newblock {\em Trans. Amer. Math. Soc.}, 361(11):6151--6163, 2009.

\bibitem{MW}
T.~Maeno and J.~Watanabe.
\newblock Lefschetz elements of {A}rtinian {G}orenstein algebras and {H}essians
  of homogeneous polynomials.
\newblock {\em Illinois J. Math.}, 53(2):591--603, 2009.

\bibitem{MMR2016}
E.~Mezzetti and R.~M. Mir\'{o}-Roig.
\newblock The minimal number of generators of a {T}ogliatti system.
\newblock {\em Ann. Mat. Pura Appl. (4)}, 195(6):2077--2098, 2016.

\bibitem{MMR2018}
E.~Mezzetti and R.~M. Mir{\'o}-Roig.
\newblock Togliatti systems and galois coverings.
\newblock {\em Journal of Algebra}, 509:263--291, 2018.

\bibitem{MMRO2013}
E.~Mezzetti, R.~M. Mir\'{o}-Roig, and G.~Ottaviani.
\newblock Laplace equations and the weak {L}efschetz property.
\newblock {\em Canad. J. Math.}, 65(3):634--654, 2013.

\bibitem{MMR}
M.~Micha{\l}ek and R.~M. Mir{\'o}-Roig.
\newblock Smooth monomial togliatti systems of cubics.
\newblock {\em Journal of Combinatorial Theory, Series A}, 143:66--87, 2016.

\bibitem{MMRN2011}
J.~C. Migliore, R.~M. Mir\'{o}-Roig, and U.~Nagel.
\newblock Monomial ideals, almost complete intersections and the weak
  {L}efschetz property.
\newblock {\em Trans. Amer. Math. Soc.}, 363(1):229--257, 2011.

\bibitem{NW}
Y.~Numata and A.~Wachi.
\newblock The strong {L}efschetz property of the coinvariant ring of the
  {C}oxeter group of type {$H_4$}.
\newblock {\em J. Algebra}, 318(2):1032--1038, 2007.

\bibitem{Ol}
P.~J. Olver.
\newblock Hyperjacobians, determinantal ideals and weak solutions to
  variational problems.
\newblock {\em Proceedings of the Royal Society of Edinburgh Section A:
  Mathematics}, 95(3-4):317--340, 1983.

\bibitem{P2000}
D.~Perkinson.
\newblock Inflections of toric varieties.
\newblock {\em Michigan Mathematical Journal}, 48(1):483--515, 2000.

\bibitem{Ru}
F.~Russo.
\newblock {\em On the geometry of some special projective varieties}.
\newblock Springer, 2016.

\bibitem{St}
R.~P. Stanley.
\newblock Hilbert functions of graded algebras.
\newblock {\em Advances in Math.}, 28(1):57--83, 1978.

\bibitem{St2}
R.~P. Stanley.
\newblock Weyl groups, the hard {L}efschetz theorem, and the {S}perner
  property.
\newblock {\em SIAM J. Algebraic Discrete Methods}, 1(2):168--184, 1980.

\bibitem{T2}
E.~Togliatti.
\newblock Alcune osservazioni sulle superficie razionali che rappresentano
  equazioni di {L}aplace.
\newblock {\em Ann. Mat. Pura Appl. (4)}, 25:325--339, 1946.

\bibitem{T1}
E.~G. Togliatti.
\newblock Alcuni esemp\^{i} di superficie algebriche degli iperspaz\^{i} che
  rappresentano un' equazione di {L}aplace.
\newblock {\em Comment. Math. Helv.}, 1(1):255--272, 1929.

\bibitem{Wa}
J.~Watanabe.
\newblock The {D}ilworth number of {A}rtinian rings and finite posets with rank
  function.
\newblock In {\em Commutative algebra and combinatorics ({K}yoto, 1985)},
  volume~11 of {\em Adv. Stud. Pure Math.}, pages 303--312. North-Holland,
  Amsterdam, 1987.

\bibitem{Za2}
F.~Zak.
\newblock Tangents and secants of algebraic varieties, transl. math.
\newblock {\em Monographs, Amer. Math. Soc}, 127, 1993.

\bibitem{Za}
F.~L. Zak.
\newblock Structure of gauss maps.
\newblock {\em Functional Analysis and Its Applications}, 21(1):32--41, 1987.

\end{thebibliography}
\bibliographystyle{abbrv}

\end{document}

\end{document}